\documentclass[final]{siamltex}

\usepackage{amsfonts}
\usepackage{amssymb,amsmath}
\usepackage[mathscr]{eucal}
\usepackage{graphicx}

\usepackage{times}          
\title{Alternative Jacobi Polynomials\\
and Orthogonal Exponentials
}

\author{Vladimir S. Chelyshkov\thanks{Lane College, 545 Lane Ave., Jackson, TN 38301 
({\tt vchelyshkov@lanecollege.edu}).}}

\begin{document}

\maketitle

\begin{abstract}
Sequences of orthogonal polynomials that are alternative to the Jacobi polynomials on the interval $[0,1]$ are defined and their properties are established. An $(\alpha,\beta)$-parameterized system of orthogonal polynomials of the exponential function on the semi-axis $[0, \infty)$ is presented. Two subsystems of the alternative Jacobi polynomials, as well as  orthogonal exponential polynomials are described. Two parameterized systems of discretely almost orthogonal functions on the interval $[0,1]$ are introduced.
\end{abstract}



\pagestyle{myheadings}
\thispagestyle{plain}
\markboth{V.~S.~CHELYSHKOV}{ALTERNATIVE JACOBI POLYNOMIALS}

\section{Introduction}

The concept of inverse orthogonalization of a sequence of functions emerged from the study of spectral methods for numerical simulation of near-wall turbulent flows \cite{ NearWall,VC2}. Turbulence in the boundary layer near a flat plate is an intensive physical phenomenon that exhibits an exponential decay of disturbances far from the wall. Thus, the sequence of exponential functions $\{e^{-ky}\}_{k=1}^n, y\in[0,\infty)$, being orthogonalized, can be employed for numerical simulation of the phenomenon. It was noticed that applying the alternative algorithm of inverse orthogonalization to the sequence of exponential functions is as constructive as applying the direct orthogonalization algorithm to the sequence of monomials $\{x^{k}\}_{k=0}^n, x \in [0,1]$, and the bidirectional algorithm of orthogonalization was introduced  in \cite{VC1} for defining systems of orthogonal exponential polynomials on the semi-axis. 

Later, the alternative Legendre polynomials were presented, and two associated quadrature rules were obtained \cite{VC4}. 

The alternative Jacobi polynomials are described in this paper. By taking
advantage of the bidirectional orthogonalization we establish general properties of the polynomials.
We show that the shifted Jacobi polynomials corresponding to the weight
function with a translated parameter are a special subset of the alternative Jacobi polynomials.
Also, we consider an exceptional case of singular orthogonality; this feature takes place
in some of the alternative systems.

We find that  the alternative Jacobi polynomials in an exponential form are the system of functions on the semi-axis $[0,\infty)$ that has the same algorithmic capacity as the Jacobi polynomials have on the interval $[-1,1]$.
Such a deduction is illustrated with two special subsystems of the exponential polynomials possessing nice properties.

The orthogonal exponential polynomials on the semi-axis may serve a useful purpose in approximation of a continuous function on a closed interval. Following \cite{VC3}, we introduce two parameterized systems of discretely almost orthogonal functions on the interval $[0,1]$.  The systems are constituted of the orthogonal exponential polynomials and unity. For the second system, supplementary computations are required to meet the definition of the functions, and properties of such a special construction have not been studied. 

To obtain these results we followed the classical theory of orthogonal polynomials. Different aspects of the theory that we were inspired by are represented in monographs \cite{Sego, Nik, Ismail}.


\section{Construction of an $\mbox{\boldmath $(\alpha,\beta)$}$-Parameterized Alternative Orthogonal Polynomials}

Let $n$ be a fixed whole number and
\begin{equation}   
\mbox{\boldmath ${\cal P}$}_{n}^{(\alpha,\beta)}(x)=\{{\cal P}_{nk}^{(\alpha,\beta)}(x)\}_{k=n}^{0}
\label{jd}
\end{equation}
is the system of polynomials defined by Rodrigues' type formula as
\begin{equation} 
{\cal P}_{nk}^{(\alpha,\beta)}(x)=\frac{x^{-\alpha-k-1}(1-x)^{-\beta}}{(n-k)!}\frac{{\mbox{d}}^{n-k}}{{\mbox{d}}x^{n-k}}(x^{\alpha+n+k+1}(1-x)^{\beta+n-k})
\label{j1}
\end{equation}
with $\alpha>-1,\quad \beta>-1$.
From (\ref{j1}) it immediately follows  that 
\[
{\cal P}_{nk}^{(\alpha,\beta)}(x)=\frac{1}{(n-k)!}\sum_{j=0}^{n-k}(-1)^j
\frac{(n-k)_{j}}{j!}
(\alpha+n+k+1)_{n-k-j}(\beta+n-k)_{j}x^{k+j}(1-x)^{n-k-j};
\]
notation $(a)_{m}$ represents the falling factorial. The last formula contains the term of lowest order in $x$ such that
\begin{equation}
{\cal P}_{nk}^{(\alpha,\beta)}(x)=\frac{\Gamma(\alpha+n+k+2)}{(n-k)!\Gamma(\alpha+2k+2)}x^{k}+ ...,
\label{j1aa}
\end{equation}
where $\Gamma (z)$ is the gamma function.

\begin{lemma}
For $0\leq k < l \leq n$
\begin{equation}
\int_0^1x^\alpha(1-x)^\beta {\cal P}_{nk}^{(\alpha,\beta)}(x)x^{l}{\rm d}x=0.
\label{j1aaa}
\end{equation}
\end{lemma}
\begin{proof}
Substituting (\ref{j1}) to (\ref{j1aaa}) and integrating by parts we confirm (\ref{j1aaa}).
\qquad\end{proof}

\begin{theorem}
The alternative Jacobi polynomials ${\cal P}_{nk}^{(\alpha,\beta)}(x)$ satisfy
\begin{equation}
\int_0^1 x^\alpha(1-x)^\beta  {\cal P}_{nk}^{(\alpha,\beta)}(x){\cal P}_{nl}^{(\alpha,\beta)}(x)
{\rm d}x=h_{nk}^{(\alpha,\beta)}\delta_{kl},
\label{j0}
\end{equation}
\begin{equation}
h_{nk}^{(\alpha,\beta)}=\frac{1}{\alpha+2k+1}\frac{\Gamma(\alpha+n+k+2)\Gamma(\beta+n-k+1)}{(n-k)!\Gamma(\alpha+\beta+n+k+2)}.
\label{j2}
\end{equation}
\end{theorem}
Here $\delta_{kl}$ is the Kronecker delta.
\begin{proof}
Since orthogonality is verified by Lemma 1.1, we only need to consider the case $l=k$. Substituting (\ref{j1}) and (\ref{j1aa}) to (\ref{j0}) and integrating by parts we calculate $h_{nk}^{(\alpha,\beta)}$. Substituting $k=0$ and $k=n$ to $h_{nk}^{(\alpha,\beta)}$ we confirm the restrictions on $\alpha$ and $\beta$ in (\ref{jd}).
\qquad\end{proof}

\begin{corollary} For $p \in {\mathbb N}$ orthogonality relations (\ref{j1aaa})~-~ (\ref{j2}) hold the invariance
\[
{h}_{n+p,k+p}^{(\alpha-2p,\beta)}={h}_{nk}^{(\alpha,\beta)},
\]
and
\[
x^p{\cal P}_{nk}^{(\alpha,\beta)}(x)={\cal P}_{n+p,k+p}^{(\alpha-2p,\beta)}(x).
\]
\end{corollary}


Polynomials  ${\cal P}_{pnk}^{(\alpha,\beta)}(x)=x^p{\cal P}_{nk}^{(\alpha,\beta)}(x)$ are orthogonal on the interval $[0,1]$ with the weight function $x^{\alpha-2p}(1-x)^{\beta}$.

Following the proof of  Theorem 2.2  we also evaluate
\[
\int_0^1x^\alpha(1-x)^\beta {\cal P}_{nk}^{(\alpha,\beta)}(x) {\rm d} x=\frac{\Gamma(\alpha+k+1)}{k!}\frac{\Gamma(\beta+n-k+1)}{(n-k)!}\frac{n!}{\Gamma(\alpha+\beta+n+2)}.
\]

Making use of (\ref{j1}), Cauchy's integral formula for analytic functions,
and reciprocal substitutions one can obtain the integral representation
\begin{equation}
\label{j21}
{\cal P}_{nk}^{(\alpha,\beta)}(x)=\frac{1}{2\pi i}\frac{(x^{-1})^{\alpha+\beta+n+2}}{(1-x^{-1})^{\beta}}\int\limits_{C}
{\frac{z^{-(\alpha+\beta+n+k+2)}(1-z)^{\beta+n-k}}{(z-x^{-1})^{n-k+1}}} {\rm d}z,
\end{equation}
where $C$ is a closed contour encircling the point $z=x^{-1}$.

\subsection{Reciprocity}

Being constructed by the procedure of inverse orthogonalization, polynomials ${\cal P}_{nk}^{(\alpha,\beta)}(x)$ can be represented in terms of the Jacobi polynomials. Below we derive two relationships of such kind.

First  we determine the relationship between $\mbox{\boldmath ${\cal P}$}_{n}^{(\alpha,\beta)}(x)$ and  polynomials from the sequence
\[
\mbox{\boldmath $P$}_{n}^{(\alpha,\beta)}(x)=\{{P}_{nk}^{(\alpha,\beta)}(x)\}_{k=n}^{\infty},
\]
that are defined by the procedure of direct orthogonalization as
\begin{equation}
\mbox{sign}\left({P}_{nk}^{(\alpha,\beta)}(1)\right)=(-1)^{k-n}, 
\label{jsign}
\end{equation}
\begin{equation}
\int_0^1 x^\alpha(1-x)^\beta  P_{nk}^{(\alpha,\beta)}(x)P_{nl}^{(\alpha,\beta)}(x)
{\rm d}x=d_{nk}^{(\alpha,\beta)}\delta_{kl},
\label{j2b}
\end{equation}
\begin{equation}
d_{nk}^{(\alpha,\beta)}=\frac{1}{\alpha+\beta+2k+1}\frac{\Gamma(\alpha+k+n+1)\Gamma(\beta+k-n+1)}{(k-n)!\Gamma(\alpha+\beta+k+n+1)}.
\label{j2a}
\end{equation}
Since $n$ is fixed,
by verifying (\ref{jsign}), (\ref{j2b}), and (\ref{j2a}),
the polynomials $P_{nk}^{(\alpha,\beta)}(x)$ can be
immediately connected to a fixed set of the Jacobi polynomials $P_m^{(\alpha
,\beta )} (x)$ \cite{Sego}. Precisely,
\begin{equation}
\label{j3}
P_{nk}^{(\alpha,\beta)} (x)=x^nP_{k-n}^{(\alpha+2n,\beta)} (1-2x).
\end{equation}
This formula can be used directly for describing properties of $P_{nk}^{(\alpha, \beta)}(x)$,
and one of the results that follow from (\ref{j3}) is the integral representation
\begin{equation}
\label{j5}
P_{nk}^{(\alpha,\beta)} (x)=\frac{1}{2\pi i}\frac{x^{-\alpha-n}}{(1-x)^{\beta}}\int\limits_{C_1}
{\frac{z^{\alpha+k+n}(1-z)^{\beta+k-n}}{(z-x)^{k-n+1}}}{\rm d}z.
\end{equation}
Here $C_1$ is a closed contour encircling the point $z=x$. Comparison of (\ref{j21}) and (\ref{j5}) leads to
\begin{equation}
{\cal P}_{nk}^{(\alpha,\beta)}(x)=x^{-1}P_{-(n+1),-(k+1)}^{(-\alpha-\beta,\beta)}(x^{-1}),\quad 0\leq k \leq n.
\label{j6}
\end{equation}


Thus, relation of reciprocity (\ref{j6}) associates two sequences of polynomials that contain different powers $x^k$ -- i.e.,
with $0\leq k \leq n$ and with $n\leq k \leq 2n$ respectively.

Let us now consider the relationship between $\mbox{\boldmath ${\cal P}$}_{n}^{(\alpha,\beta)}(x)$ and polynomials from the sequence $\mbox{\boldmath $P$}_{0}^{(\alpha,\beta)}(x)$ such that 
\begin{equation}
{P}_{0k}^{(\alpha,\beta)}(x)=P_{k}^{(\alpha,\beta)}(1-2x).
\label{jj1}
\end{equation}
Shifted Jacobi polynomials (\ref{jj1}) are orthogonal on the interval $[0,1]$ with the weight function $x^\alpha(1-x)^\beta$. The polynomials can also be defined by the Rodrigues formula as
\begin{equation}
{P}_{0k}^{(\alpha,\beta)}(x)=\frac{x^{-\alpha}(1-x)^{-\beta}}{k!}\frac{{\rm d}^{k}}{{\rm d}x^{k}}(x^{\alpha+k}(1-x)^{\beta+k}).
\label{jj2}
\end{equation}
Comparing (\ref{j1}) and (\ref{jj2}) we find that
\begin{equation}
{\cal P}_{nk}^{(\alpha,\beta)}(x)=x^{k}{P}_{0,n-k}^{(\alpha+2k+1,\beta)}(x).
\label{jj3}
\end{equation}

\subsection{Recurrence Relations and Properties Involving Differentiation}

Relation of reciprocity (\ref{j6}) and formula (\ref{j3}), as well as (\ref{jj3}) and (\ref{jj1}), enable employing the classical results for establishing properties of ${\cal P}_{nk}^{(\alpha,\beta)}(x)$. 

Following  (\ref{j6}) and (\ref{j3}) we obtain the three-term recurrence relation
\[
{\cal P}_{nn}^{(\alpha,\beta)}(x)=x^{n},
\]
\[
{\cal P}_{n,n-1}^{(\alpha,\beta)}(x)=(\alpha+2n)x^{n-1}-(\alpha+\beta+2n+1)x^{n},
\]
\[
(n-k+1)(\alpha+n+k+1)(\alpha+2k+2){\cal P}_{n,k-1}^{(\alpha,\beta)}(x)=(\alpha+2k+1)
\]
\[
\times[(\alpha+2k)(\alpha+2k+2)x^{-1}-(\alpha+2n+2)(\alpha+\beta+2k+1)-2(n-k)(n-k+1)]
\]
\[
\times{\cal P}_{nk}^{(\alpha,\beta)}(x)-(\alpha+\beta+n+k+2)(\beta+n-k)(\alpha+2k){\cal P}_{n,k+1}^{(\alpha,\beta)}(x)
\]
and the differential-difference relations 
\[
(\alpha+2k+2)x(1-x)\frac{\mbox{d}}{\mbox{d}x}{\cal P}_{n,k}^{(\alpha,\beta)}(x)=[k\alpha+2k(k+1)-(n\alpha+n^2+k^2+2n)x]
\]
\[
\times x{\cal P}_{nk}^{(\alpha,\beta)}(x)-(\alpha+\beta+n+k+2)(\beta+n-k)x{\cal P}_{n,k+1}^{(\alpha,\beta)}(x),
\]
\[
(\alpha+2k)x(1-x)\frac{\rm d}{{\rm d}x}{\cal P}_{n,k}^{(\alpha,\beta)}(x)=[-(\alpha+2k)(\alpha+k+1)-\beta(\alpha+2k)
\]
\[
+((n+1)(\alpha+\beta+n+1)+(\alpha+k)(\alpha+\beta+k)+\alpha+2k)x]x{\cal P}_{nk}^{(\alpha,\beta)}(x)
\]
\[
+(n-k+1)(\alpha+n+k+1)x{\cal P}_{n,k-1}^{(\alpha,\beta)}(x).
\]
Following  (\ref{jj3}) and (\ref{jj1}) we find the differentiation formula
\begin{equation}
\label{dx}
\frac{\rm d}{{\rm d}x}{\cal P}_{nk}^{(\alpha,\beta)}(x)-kx^{-1}{\cal P}_{nk}^{(\alpha,\beta)}(x)=-(\alpha + \beta +n+k+2){\cal P}_{n-1,k}^{(\alpha+1,\beta+1)}(x)
\end{equation}
for  $k<n$, and the differential equation 
\[
x^2(1-x)y^{\prime\prime}(x)+x[\alpha+2-(\alpha+\beta+3)x]y^{\prime}(x)
\]
\begin{equation}
-[k(\alpha+k+1)-n(\alpha+\beta+n+2)x]y(x)=0
\label{de1}
\end{equation}
that has a solution $y={\cal P}_{nk}^{(\alpha,\beta)}(x)$.

Let us consider the case $k=0$. It follows  from (\ref{jj3}) that
\begin{equation}
{\cal P}_{n0}^{(\alpha,\beta)}(x)={P}_{0n}^{(\alpha+1,\beta)}(x), \quad \alpha>-2, \quad \beta>-1.
\label{al2}
\end{equation}
Thus, the shifted Jacobi polynomials with translated $\alpha$,  ${P}_{0n}^{(\alpha+1,\beta)}(x)$, are a subset of the alternative Jacobi polynomials 
${\cal P}_{nk}^{(\alpha,\beta)}(x)$. Relation (\ref{al2}) can also be directly traced from equation  (\ref{de1}). For $k=0$ the degrees of the polynomial coefficients in (\ref{de1}) are reduced by one, and this makes ${\cal P}_{n0}^{(\alpha,\beta)}(x)$ the sequence with classical properties. 

For $k>0$ polynomials ${\cal P}_{nk}^{(\alpha,\beta)}(x)$ are less perfect. By considering property (\ref{dx}), or by differentiating  (\ref{de1}), we find that the first derivative of a ${\cal P}_{nk}^{(\alpha,\beta)}(x)$ 
is not a constant multiple of a ${\cal P}_{n-1,k-1}^{(\alpha^{\prime},\beta^{\prime})}(x)$ with an $(\alpha^{\prime},\beta^{\prime})$, and so $d{\cal P}_{nk}^{(\alpha,\beta)}(x)/dx$ do not belong to the original class of polynomials. Still, other properties of  ${\cal P}_{nk}^{(\alpha,\beta)}(x)$ described above are similar to properties of the classical orthogonal polynomials.

\section{Singular Orthogonality and Marginal Sequences}

The restriction $\alpha>-1$ for the the weight function $w(x;\alpha,\beta)=x^{\alpha}(1-~x)^{\beta}$ in (\ref{j1aaa}), (\ref{j0})  is stronger than the restriction on $\alpha$ in (\ref{al2}). Let $\alpha^{\prime}=\{\alpha :\quad  -2<\alpha\leq-1\}$. In this exceptional case the polynomial ${\cal P}_{n0}^{(\alpha^{\prime},\beta)}(x)$ is not integrable with respect to the weight function $w(x;\alpha^{\prime},\beta)$
and therefore non-normalizable. However,  it is consistent with orthogonality relations (\ref{j1aaa}),  (\ref{j0}) for $l>~0$. Thus,  ${\cal P}_{n0}^{(\alpha^{\prime},\beta)}(x)$   is a singular term of the sequence $\mbox{\boldmath ${\cal P}$}_{n}^{(\alpha^{\prime},\beta)}(x)$, which we call a marginal system. 

Assuming $k>0$ one can employ the sequences ${\cal P}_{nk}^{(\alpha^{\prime},\beta)}(x)$ for approximation of a continuous function $f(x)$ on the interval $[0,1]$. Such an approximation is effective near $x=0$ only if $f(0)=0$. 

Below we describe two noteworthy orthogonal sequences that hold singular terms. 

\subsection{A-Kind Orthogonal Polynomials} 

Considering (\ref{j1}) we find  that
${\cal P}_{n0}^{(-1,0)}(x)$ are the shifted Legendre polynomials. This classical set of polynomials is a singular subset of the system  
\begin{equation}
\mbox{\boldmath ${\mathscr A}$}_{n}(x)=\mbox{\boldmath ${\cal P}$}_{n}^{(-1,0)}(x), \quad
\mbox{\boldmath ${\mathscr A}$}_{n}(x)=
\{{\mathscr A}_{nk}(x)\}_{k=n}^{0}, \quad n \in  {\mathbb N}.
\label{lk0}
\end{equation}
The system $\mbox{\boldmath ${\mathscr A}$}_{n}(x)$ was introduced in \cite{VC1}, and it is the first example of an inversely orthogonalized  sequence of monomials $\{x^k\}_{k=0}^n$.

The polynomials ${\mathscr A}_{nk}(x)$ obey the orthogonality relations
\[
\int_0^1\frac{1}{x}
{\mathscr A}_{nk}(x)
{\mathscr A}_{nl}(x){\rm d}x=\frac{\delta_{kl}}{k+l},\quad k=0,1,...,n,\quad l=1,2,...,n
\]
and 
\[
\int_0^1\frac{1}{x}
{\mathscr A}_{nk}(x){\rm d}x=\frac{1}{k},\quad k=1,...,n.
\]
\begin{figure}[htbp]
\centerline{\includegraphics[height=36mm, width=60mm]{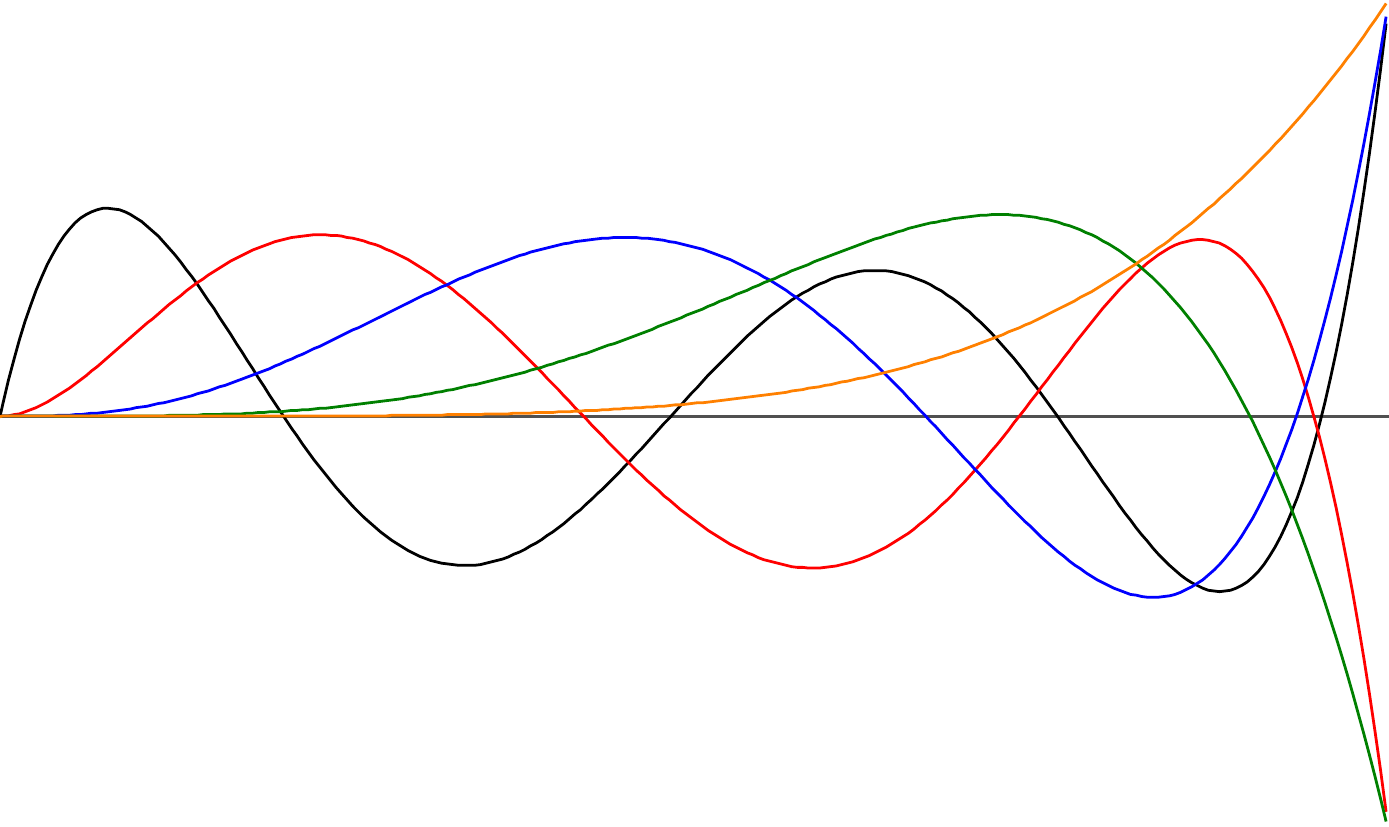}}
\caption{A-kind orthogonal polynomials: $n=5, k=1-5$.}
\normalsize
\end{figure}
They can be calculated using the expansion
\[
{\mathscr A}_{nk} (x)=\sum\limits_{j=0}^{n-k} {(-1)^j} \left( {\begin{array}{c}
 n-k\\
\mbox{ }j \\
\end{array}} \right)\left({\begin{array}{c}
 n+k+j\\
\mbox{ }n-k\\
\end{array}}\right)x^{k+j},\mbox{ }k=0,1,...,n,
\]
or following the three-term recurrence relation
\[
{\mathscr A}_{nn}(x)=x^{n},\quad {\mathscr A}_{n,n-1}(x)=(2n-1)x^{n-1}-2nx^{n},
\]
\[
(2k+1)(n+k)(n-k+1){\mathscr A}_{n,k-1}(x)
\]
\[
=2k[(2k-1)(2k+1)x^{-1}-2(n^{2}+k^{2}+n)]{\mathscr A}_{nk}(x)
\]
\[
-(2k-1)(n-k)(n+k+1){\mathscr A}_{n,k+1}(x).
\]

They satisfy the differential-difference relations
\[
(2k+1)x(1-x)\frac{\rm d}{{\rm d}x}{\mathscr A}_{n,k}(x)
\]
\[
=[2k^2+1-(n^{2}+k^{2}+2n)x]x{\mathscr A}_{nk}(x)-(n-k)(n+k+1)x{\mathscr A}_{n,k+1}(x),
\]
and
\[
(2k-1)x(1-x)\frac{\rm d}{{\rm d}x}{\mathscr A}_{n,k}(x)
\]
\[
=[-k(2k-1)+(n^{2}+k^{2}+n)x]x{\mathscr A}_{nk}(x)+(n+k)(n-k+1)x{\mathscr A}_{n,k-1}(x).
\]
We also find that $y(x)={\mathscr A}_{n,k}(x)$ is a solutions to the differential equation
\[
x^{2}(1-x)y^{\prime\prime}(x)+x(1-2x)y^{\prime}(x)-[k^{2}-n(n+1)x]y(x)=0.
\]

\subsection{T-Kind Orthogonal Polynomials} 

From (\ref{j1}) it follows that 
$n!/(n-1/2)_{n}\times{\cal P}_{n0}^{(-\frac{3}{2},-\frac{1}{2})}(x)$ are the (shifted) Chebyshev  polynomials. This classical set of highly demanded in approximation theory orthogonal polynomials is a singular subset of the system of polynomials
\begin{equation}
\mbox{\boldmath ${\mathscr T}$}_{n}(x)=\{{\mathscr T}_{nk}(x)\}_{k=n}^{0},\quad
{\mathscr T}_{nk}(x)=c_{nk}{\cal P}_{nk}^{(-\frac{3}{2},-\frac{1}{2})}(x),
\label{tlk0}
\end{equation}
where $c_{nk}=(n-k)!/(n-k-1/2)_{n-k}$ and $n \in  {\mathbb N}$.

The polynomials obey the orthogonality relation
\begin{equation}
\label{tlk1}\quad
\int_0^1\frac{{\mathscr T}_{nk}(x){\mathscr T}_{nl}(x)
}{x\sqrt{x(1-x)}}{\rm d}x=\frac{1}{2(k+l)-1}\frac{(2n-k-l)!!}{(2n-k-l-1)!!}\frac{(2n+k+l-1)!!}{(2n+k+l-2)!!}\pi\delta_{kl},
\end{equation}
\[
k=0,1,...,n,\quad l=1,2,...,n.
\]
 and
\[
\int_0^1\frac{{\mathscr T}_{nk}(x)
}{x\sqrt{x(1-x)}}{\rm d}x=\frac{(2k-3)!!}{(2k)!!}2n\pi,\quad k=1,2...n.
\]
\begin{figure}[htbp]
\centerline{\includegraphics[height=36mm, width=60mm]{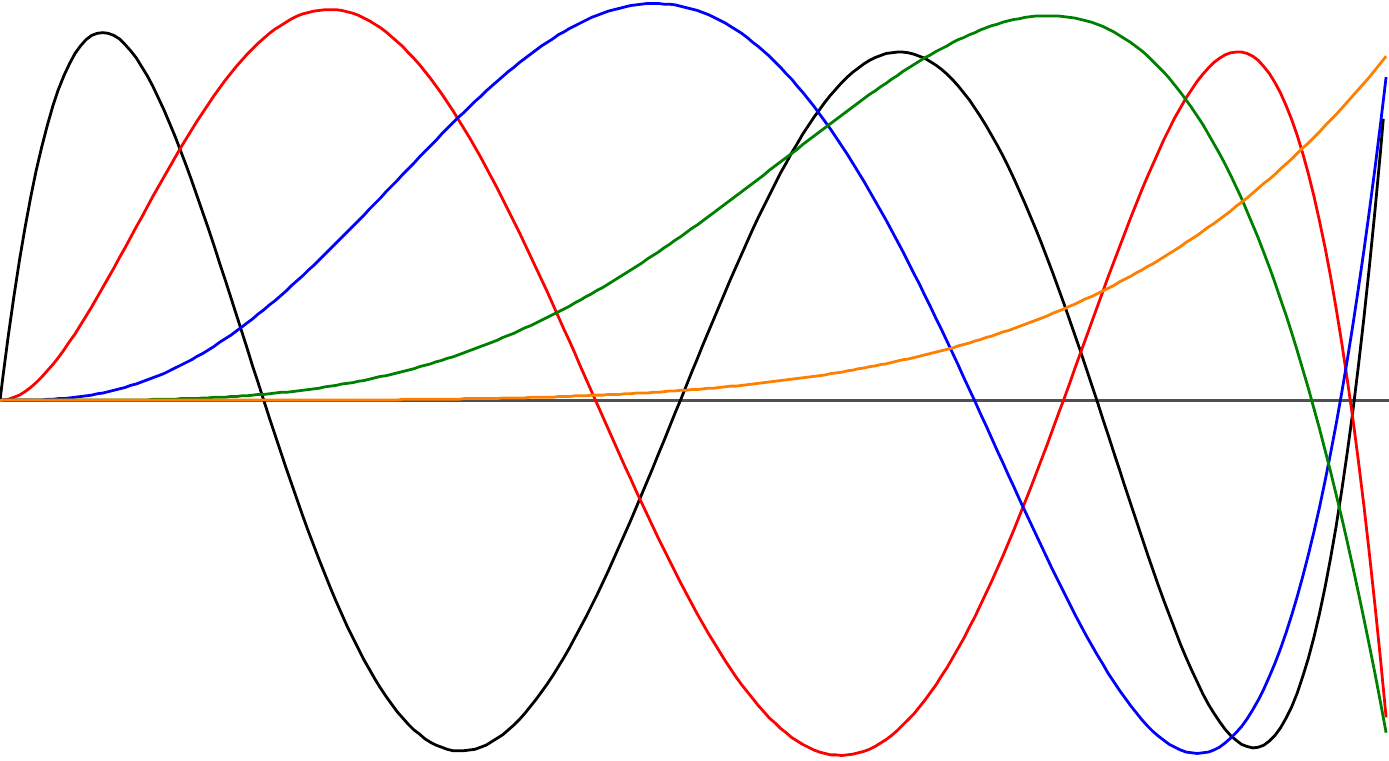}}
\caption{T-kind orthogonal polynomials: $n=5, k=1-5$.}
\normalsize
\end{figure}
They can be calculated using the three-term recurrence relation
\[
{\mathscr T}_{nn}(x)=x^{n},\quad {\mathscr T}_{n,n-1}(x)=(4n-3)x^{n-1}-(4n-2)x^{n},
\]
\[
[2(n+k)-1][2(n-k)+1](4k+1){\mathscr T}_{n,k-1}(x)
\]
\[
=(4k-1)[(4k-3)(4k+1)x^{-1}-2(4n^{2}+4k^{2}-2k-1)]{\mathscr T}_{nk}(x)
\]
\[
-4(n-k)(n+k)(4k-3){\mathscr T}_{n,k+1}(x).
\]
We also find that $y(x)={\mathscr T}_{n,k}(x)$ is a solution to the differential equation
\[
x^2(1-x)y^{\prime\prime}(x)+x(1/2-x)y^{\prime}(x)-[k(k-1/2)-n^2x]y(x)=0.
\]


\section{An  $\mbox{\boldmath $(\alpha,\beta)$}$-Parameterized Orthogonal Exponential Polynomials on the Semi-Axis}

The classical orthogonal polynomials have a property of discrete orthogonality and the alternative orthogonal polynomials have it also  (see, for example, \cite{VC4} \footnote[1]{Formula (3.5) on page 21 in \cite{VC4} have to be finalized as
\[
w_s=-\frac{2}{n(n+2)}\cdot \frac{1}{x_s^2{\cal P}_{n1} (x_s){\cal P}'_{n0} (x_s)}.
\]
}). 
Let us consider the sequence 
\begin{equation}
\mbox{exp}(-kt),\quad k \in {\mathbb N}.
\label{e00}
\end{equation}
Since 1 is not the term of the sequence, direct orthogonalization of (\ref{e00}) on the semi-axis with a chosen weight function does not result in Gauss' type quadrature and, consequently, in the discrete orthogonality of the sequence. Thus, inverse orthogonalization is the only procedure for constructing the exponential polynomials possessing such a property (see, for example, \cite{VC2}). This makes the set of orthogonal functions introduced below a unique system.

Let $\mbox{\boldmath ${\cal E}$}_{n}^{(\alpha,\beta)}(t)=\{{\cal E}_{nk}^{(\alpha,\beta)}(t)\}_{k=n}^1$ be the polynomials of exponential function defined as
\begin{equation}
{\cal E}_{nk}^{(\alpha,\beta)}(t)={\cal P}_{nk}^{(\alpha-1,\beta)}(e^{-t}).
\label{e0}
\end{equation}
From (\ref{e0}) it immediately follows that $\mbox{\boldmath ${\cal E}$}_{n}^{(\alpha,\beta)}(t)$ is the orthogonal system, such that
\[
\int_0^{\infty} e^{-\alpha t}(1-e^{-t})^\beta  {\cal E}_{nk}^{(\alpha,\beta)}(t){\cal E}_{nl}^{(\alpha,\beta)}(t)
{\rm d}t
\]
\begin{equation}
=\frac{1}{\alpha+2k}\frac{\Gamma(\alpha+n+k+1)\Gamma(\beta+n-k+1)}{(n-k)!\Gamma(\alpha+\beta+n+k+1)}\delta_{kl}
\label{e1}
\end{equation}
for $\alpha>-1$, $\beta>-1$. 

Properties of ${\cal E}_{nk}^{(\alpha,\beta)}(t)$ 
that are of interest for computations can be easily derived from the results of  previous sections. In particular,
(\ref{e1}) leads to the three-term recurrence relationship,  
which sets the associated orthogonal function ${\cal E}_{n0}^{(\alpha,\beta)}(t)$. Since ${\cal E}_{n0}^{(\alpha,\beta)}(t)$  is not integrable on the semi-axis with the weight function 1, we do not include it in $\mbox{\boldmath ${\cal E}$}_{n}^{(\alpha,\beta)}(t)$.
Functions $\left(\mbox{\boldmath ${\cal E}$}_{n}^{(\alpha,\beta)}\cup{\cal E}_{n0}^{(\alpha,\beta)}\right)(t)$ form the system with nice properties, because the zeros of  ${\cal E}_{n0}^{(\alpha,\beta)}(t)$ are the abscissas of Gauss' type quadratures for the exponential functions on the semi-axis.

In the next subsection we present two systems that are the exponential form of polynomials ${\mathscr A}_{nk}(x)$ and ${\mathscr T}_{nk}(x)$.

\subsection{Two Systems}

Let $\alpha=\beta=0$. Following (\ref{e0}) and (\ref{lk0}) we define
\[
{\cal E}_{nk}(t)={\mathscr A}_{nk}(e^{-t}).
\]
The orthogonal exponential polynomials ${\cal E}_{nk}(t)$ were first introduced in \cite{VC1}. Since they are of interest for approximation of a function in $L_2[0,\infty)$,  we repeat some properties of ${\mathscr A}_{nk}(x)$ in the exponential form as follows
\[
\int_0^{\infty}{\cal E}_{nk}(t){\cal E}_{nl}(t){\rm d}t=\frac{1}{2k}\delta_{kl},\quad k,l=1,...,n,
\]
\[
\int_0^{\infty}{\cal E}_{nk}(t){\rm d}t=\frac{1}{k},\quad k=1,...,n,
\]
\[
{\cal E}_{nn}(t)=e^{-nt},
\quad {\cal E}_{n,n-1}(t)=(2n-1)e^{-(n-1)t}
-2ne^{-nt},
\]
\[
(2k+1)(n+k)(n-k+1){\cal E}_{n,k-1}(t)
\]
\[
=2k[(2k-1)(2k+1)e^t
-2(n^{2}+k^{2}+n)]{\cal E}_{nk}(t)
\]
\[
-(2k-1)(n-k)(n+k+1){\cal E}_{n,k+1}(t),\quad k=n-1,...,1,
\]
\[
\frac{\rm d}{{\rm d}t}[{\cal E}_{n,k-1}(t)+{\cal E}_{n,k}(t)]=-(k-1){\cal E}_{n,k-1}(t)+k{\cal E}_{nk}(t).
\]

Let $\alpha=\beta=-1/2$.  Following (\ref{e0}) and (\ref{tlk0}) we define
\[
{\cal E}_{nk}^{\mathscr T}(t)={\mathscr T}_{nk}(e^{-t}).
\]
For convenience, we also repeat some properties of ${\mathscr T}_{nk}(x)$ in the exponential form as follows
\[
\int_0^\infty\frac{{\cal E}_{nk}^{\mathscr T}(t){\cal E}_{nl}^{\mathscr T}(t)
}{\sqrt{e^{-t}(1-e^{-t})}}{\rm d}t=\frac{1}{4k-1}\frac{(2(n-k))!!}{(2(n-k)-1)!!}\frac{(2(n+k)-1)!!}{(2(n+k)-2)!!}\pi\delta_{kl},\quad k,l=1,...,n,
\]
\[
\int_0^{\infty}\frac{{\cal E}_{nk}^{\mathscr T}(t)
}{\sqrt{e^{-t}(1-e^{-t})}}{\rm d}t=\frac{(2k-3)!!}{(2k)!!}2n\pi,\quad k=1,...,n,
\]
\[
{\cal E}_{nn}^{\mathscr T}(t)=e^{-nt},\quad {\cal E}_{n,n-1}^{\mathscr T}(t)=(4n-3)e^{-(n-1)t}-(4n-2)e^{-nt},
\]
\[
[2(n+k)-1][2(n-k)+1](4k+1){\cal E}_{n,k-1}^{\mathscr T}(t)
\]
\[
=(4k-1)[(4k-3)(4k+1)e^{t}-2(4n^{2}+4k^{2}-2k-1)]{\cal E}_{nk}^{\mathscr T}(t)
\]
\[
-4(n-k)(n+k)(4k-3){\cal E}_{n,k+1}^{\mathscr T}(t),\quad k=n-1,...,1.
\]
The zeros of ${\cal E}_{n0}^{\mathscr T}(t)$ is easy to calculate,  and this advantage can be a motivation for approximation of a function on the semi-axis by ${\cal E}_{nk}^{\mathscr T}(t)$. 

\subsection{Discretely Almost Orthogonal Functions on the Interval $[0,1]$} 

Employing exponential polynomials ${\cal E}_{nk}^{(\alpha, \beta)}(t)$ can be of interest for approximation of a function on the interval $[0,1]$.

One may try to apply the system 
\[  
\mbox{\boldmath $\widetilde{{\cal E}}$}_{n}^{(\alpha,\beta)}(t)=\{1, \widetilde{{\cal E}}_{nk}^{(\alpha,\beta)}(t)\}_{k=1}^n,\quad
\widetilde{{\cal E}}_{nk}^{(\alpha,\beta)}(t)={\cal E}_{nk}^{(\alpha,\beta)}(\lambda_{nn}^{(\alpha,\beta)}\cdot t).
\] 
as the basis functions. Here, the scale parameter $\lambda_{nn}^{(\alpha,\beta)}$ is the maximum non-trivial zero of the zeros $\lambda_{nk}^{(\alpha,\beta)}$ of the function ${\cal E}_{n0}^{(\alpha, \beta)}(t)$  \cite{VC3}.

Another way is the use of parameters $(\alpha,\beta)$. Below we describe a formal construction of   functions that involves computation of  $(\alpha_n,\beta_n)$ for desirable distribution of the zeros  $\lambda_{nk}^{(\alpha_n,\beta_n)}$ on the interval $(0,1]$.

First, by considering $\beta=\omega\alpha$, we introduce parameter $\omega$ that sets the abscissa of the weight function maximum, $t=\mbox{ln}(1+\omega)$. Then, by subjecting choice of $\alpha_n$ to the requirements
\begin{equation}
\label{A1}
\gamma_n\leq 1\quad \mbox{and} \quad \gamma_n=\underset{\alpha_n}\max\hspace{2pt} \lambda_{nn}^{(\alpha_n, \omega_n\alpha_n)},
\end{equation}
we introduce parameter $\gamma_n$.
In general, the equality in (\ref{A1}) can be reached if $\alpha_n,\beta_n \in {\mathbb R}$; the inequality holds if $\alpha_n,\beta_n$ are selected from  the set of whole numbers or from a set of rational numbers of interest. Accordingly, for $\gamma_n=1$ we define the system of functions
\[  
\mbox{\boldmath ${\cal Z}$}_{n}^{(\omega_n)}(t)=\{1, {\cal Z}_{nk}^{(\omega_n)}(t)\}_{k=1}^n,\quad
{\cal Z}_{nk}^{(\omega_n)}(t)={\cal E}_{nk}^{(\alpha_n, \omega_n\alpha_n)}(t),
\]
and, for $\gamma_{n}<1$,
\[  
\mbox{\boldmath $\widetilde{{\cal Z}}$}_{n}^{(\omega_n)}(t)=\{1, \widetilde{{\cal Z}}_{nk}^{(\omega_n)}(t)\}_{k=1}^n,\quad
\widetilde{{\cal Z}}_{nk}^{(\omega_n)}(t)={\cal E}_{nk}^{(\alpha_n, \omega_n\alpha_n)}(\gamma_n t).
\] 
Similarly, we define the sequences ${\cal Z}_{n0}^{(\omega_n)}(t)$ and $\widetilde{{\cal Z}}_{n0}^{(\omega_n)}(t)$. Choice of $\omega_n$ sets distributions of the zeros of the functions under consideration, 
and ${\cal Z}_{n0}^{(\omega_n)}(1)=\widetilde{{\cal Z}}_{n0}^{(\omega_n)}(1)=0$. The system of functions $\left(\mbox{\boldmath $\widetilde{{\cal Z}}$}_{n}^{(\omega_n)} \cup \widetilde{{\cal Z}}_{n0}^{(\omega_n)}\right)(t)$
is easier to construct, in particular for $\alpha_n,\beta_n \in {\mathbb N}$.

An interpolating sequence of functions corresponding to approximation by Z-functions with $n=2$, $\omega_2=1$ and $\alpha_2,\beta_2 \in {\mathbb N}$ was employed  in \cite{VC3} for solving a stiff initial value problem.

\section{Conclusion} 
Making use of the algorithm of bidirectional orthogonalization we obtained two systems of orthogonal functions, i.e. the alternative Jacobi polynomials
${\cal P}_{nk}^{(\alpha,\beta)}(x)$ and the orthogonal exponential polynomials ${\cal E}_{nk}^{(\alpha, \beta)}(t)$. Algorithmic properties of the systems are very similar to   the properties of the Jacobi polynomials. This indicates that  each of the systems  may be a convenient tool for applied analysis.

In addition, we described two marginal sequences  
${\mathscr A}_{nk}(x)$ 
and  
${\mathscr T}_{nk}(x)$, 
as well as their exponential counterparts  
${\cal E}_{nk}(t)$ 
and  
${\cal E}_{nk}^{\mathscr T}(t)$. 
Similarly, other important special systems of polynomials can be derived from the alternative Jacobi polynomials. 

Adapted systems  of functions $\mbox{\boldmath $\widetilde{{\cal E}}$}_{n}^{(\alpha,\beta)}(t)$ and $\mbox{\boldmath ${\cal 
 Z}$}_{n}^{(\omega_n)}(t)$ and their extensions that are appropriate for differentiation may be of interest for {\em hp}-approximation on a finite interval.

It should be mentioned that there are regular orthogonal exponential polynomials that are often used for approximation on the semi-axis, but usually they are not considered as  special functions sui generis (say, in \cite{Spalart}).  Description of these orthogonal exponentials can be found in \cite{Jaroch}. Also, a special orthogonal sequence of exponentials that are defined on the semi-axis in a  particularly appealing form is constructed in \cite{Baxter}.

The concept that is developed in this paper is quite consistent with the maxim of C.G.J.~Jacobi: ``~Invert, always invert~". Many years ago we unknowingly subscribed to this dictum, and now we recognize its notability \cite{Nashed}.

\end{document}